\documentclass{aims}
\usepackage{amsmath}
  \usepackage{paralist}
 \usepackage[colorlinks=true]{hyperref}
\hypersetup{urlcolor=blue, citecolor=red}

\def\currentvolume{X}
 \def\currentissue{X}
  \def\currentyear{200X}
   \def\currentmonth{XX}
    \def\ppages{X--XX}

\newtheorem{thm}{Theorem}[section]
\newtheorem{cor}{Corollary}
\newtheorem{lem}[thm]{Lemma}
\newtheorem{prop}{Proposition}
\theoremstyle{definition}
\newtheorem{defn}[thm]{Definition}
\newtheorem{rem}{Remark}
\newtheorem{example}[thm]{Example}

\newcommand{\hilf}{\mathcal{H}}   
\newcommand{\R}{\mathbb{R}}   

\newcommand{\A}{\mathcal{A}}
\newcommand{\h}{\mathcal{H}}
\def\a{\alpha}
\def\o{\omega}
\def \O {\Omega}

\title [Indirect stabilization with hybrid boundary conditions] 
{Indirect stabilization of weakly coupled systems with hybrid boundary conditions}

\author[F. Alabau-Boussouira, P. Cannarsa and R. Guglielmi]{}


 \email{alabau@univ-metz.fr}
 \email{cannarsa@mat.uniroma2.it}
 \email{guglielm@mat.uniroma2.it}

\thanks{The authors wish to thank Institut Henri Poincar\'e (Paris, France) for providing a very stimulating environment during the "Control of Partial and Differential Equations and Applications" program in the Fall 2010. Part of this work was completed during the C.I.M.E. course `Control of partial differential equations' (Cetraro, July 19--23, 2010). This research has been performed in the framework of the GDRE CONEDP, see \mbox{\url{http://www.ceremade.dauphine.fr/~glass/GDRE/index.php}}}

\begin{document}
\maketitle

\centerline{\scshape Fatiha Alabau-Boussouira}
\medskip
{\footnotesize
 \centerline{Present position D\'el\'egation CNRS at MAPMO, UMR 6628.}
   \centerline{ Current position Universit\'e Paul Verlaine-Metz,}
   \centerline{ Ile du Saulcy, 57045 Metz Cedex 1, France}
} 

\medskip

\centerline{\scshape Piermarco Cannarsa and Roberto Guglielmi}
\medskip
{\footnotesize
 \centerline{Dipartimento di Matematica - Universit\`a di Roma ``Tor Vergata''}
   \centerline{Via della Ricerca Scientifica 1 - 00133 Roma, ITALY}
}

\bigskip

 \centerline{(Communicated by the associate editor name)}

\begin{abstract}
We investigate stability properties of indirectly damped systems of evolution equations in Hilbert spaces, under new compatibility assumptions. We prove polynomial decay for the energy of solutions and optimize our results by interpolation techniques, obtaining a full range of power-like decay rates. In particular, we give explicit estimates with respect to the initial data. We discuss several applications to hyperbolic systems with {\em hybrid} boundary conditions, including the coupling of two wave equations subject to Dirichlet and Robin type boundary conditions, respectively.
\end{abstract}

\section{Introduction}\label{se:intro}

There is no doubt that the interest of the scientific community in the stabilization and control of systems of partial differential equations has remarkably increased, in recent years. This is probably due to the fact that such systems arise in several applied mathematical models, such as those used for studying the vibrations of flexible structures and networks (see \cite{DZ} and references therein), or fluids and fluid-structure interactions (see, for instance, \cite{AT}, \cite{ALT}, \cite{BO}, \cite{IT}, \cite{RV}, \cite{ZZ}).

When dealing with systems involving quantities described by several components, pretending to control or observe all the state variables might be irrealistic. In applications to mathematical models for the vibrations of flexible structures (see \cite{A1} and \cite{AMR}), electromagnetism (see, for instance, \cite{Ka}), or fluid control (see \cite{CG} and the references therein), it may happen that only part of such components can be observed. This is why it becomes essential to study whether controlling only a reduced number of state variables suffices  to ensure the stability of the full system. 

It turns out that certain systems possess an internal structure that compensates for the aforementioned lack of control variables. Such a phenomenon is referred to as {\em indirect stabilization} or {\em indirect control} (see \cite{Ru}). 
An example of indirect stabilization occurs with the hyperbolic system
\begin{equation}\label{IS}
\begin{cases}
\partial_t ^2u-\Delta  u + \partial_t u+\a v=0
&\mbox{in}\quad\Omega\times\R
\\
\partial_t ^2v -\Delta  v +\a u=0
& \mbox{in}\quad\Omega\times\R
\\
u =0=v & \mbox{on}\quad\partial\Omega\times\R\,,
\end{cases}
\end{equation}
where $\Omega$ is a bounded open domain of $\mathbb{R}^N$, and the `frictional' term
 $\partial_t u$ acts as a stabilizer. Indeed, a general result proved in \cite{alcako02} ensures that, for sufficiently smooth initial conditions and $|\a|>0$ small enough, the energy of the solution $(u,v)$ of \eqref{IS} decays to zero at a polynomial rate as $t\to\infty$.
 
The above indirect stabilization property holds true for more general systems of partial differential equations, under the compatibility assumption \eqref{theta2a} below, see \cite{alcako02}. For applications to problems in mechanical engineering, however, it is extremely important to consider also boundary conditions that fail to satisfy the assumption of  \cite{alcako02}. This is the case of Neumann or Robin boundary conditions, which describe different physical situations such as hinged or clamped devices.   
For instance, let us change the boundary conditions in \eqref{IS} as follows:
\begin{equation}\label{IS1}
\begin{cases}
\partial_t ^2u-\Delta  u + \partial_t u+\a v=0
&\mbox{in}\quad\Omega\times\R
\\
\partial_t ^2v -\Delta  v +\a u=0
& \mbox{in}\quad\Omega\times\R
\\
u + \frac{\partial u}{\partial \nu} =0=v & \mbox{on}\quad\partial\Omega\times\R\, .
\end{cases}
\end{equation}
Then, as is shown in Proposition \ref{compcondns} below, the compatibility assumption \eqref{theta2a} is not satisfied. Nevertheless, in this paper we will prove polynomial stability for system \eqref{IS1}, using a new hypothesis which is specially designed to handle boundary conditions as above---that we call {\em hybrid}. 

More generally, in a real Hilbert space $H$, with scalar product $\langle\cdot,\cdot\rangle$ and norm $|\,\cdot\,|$, we shall study the system of evolution equations
\begin{equation}\label{eq:1z}
\begin{cases}
u''(t)+A_1u(t)+Bu'(t)+\a v(t)=0
\\
v''(t)+A_2v(t)+\a u(t)=0
\end{cases}
\end{equation}
where
\begin{enumerate}
\item [(H1)]
 $A_i:D(A_i)\subset H\to H\;(i=1,2)$ are densely defined closed linear operators such that
 \begin{eqnarray*}
A_i=A_i^*\,,\qquad
\langle A_iu,u\rangle\geq \omega_i |u|^2\qquad\forall u\in D(A_i)
\end{eqnarray*}
for some $\omega_1,\omega_2>0$,
\item[(H2)]
$B$ is a bounded linear operator on $H$ such that
 \begin{eqnarray*}
B=B^*\,,\qquad
\langle Bu,u\rangle\geq \beta |u|^2\qquad\forall u\in H
\end{eqnarray*}
for some $\beta>0$,
\item[(H3)]
$\a$ is a real number such that
$$
0<|\a|<\sqrt{\o_1\o_2}\,.
$$
\end{enumerate}
System (\ref{eq:1z}), with the initial conditions
\begin{equation}\label{eq:icz}
\left\{\begin{array}{ll}
u(0)=u^0\,,&\quad u'(0)=u^1\, ,
\\
v(0)=v^0\,,&\quad v'(0)=v^1\,,
\end{array}\right.
\end{equation}
can be formulated as a Cauchy problem for a certain first order evolution equation in the product space
$$
\hilf:=D(A_1^{1/2})\times H\times D(A_2^{1/2})\times H \,.
$$
More precisely, let us define the energies associated to operators $A_1,A_2$ by 
\begin{equation}\label{eq:enercompsz}
E_i(u,p)=\frac12 \left(|A_i^{1/2}u|^2+|p|^2\right)\quad \forall (u,p)\in
D(A_i^{1/2})\times H\ (i=1,2)\,,
\end{equation}
and the total energy of the system as
\begin{equation}
\label{calEz}
{\mathcal E}(U):=E_1(u,p)+E_2(v,q)+\a \langle u,v\rangle
\end{equation}
for every $U = \left(u,p,v,q\right)\in\hilf$. Then, assumption (H1) yields, for $i=1,2$,
\begin{equation}\label{2.4z}
|u|^2\le \frac2{\omega_i}\,E_i(u,p)\qquad\forall u\in D(A_i^{1/2}),
\;\forall p \in H\,.
\end{equation}
Moreover, in view of $(H3)$, for all $U=(u,p,v,q)\in \h$
\begin{equation}\label{2.6z}
{\mathcal E}(U)\ge
\nu(\a)\Big[E_1(u,p)+E_2(v,q)\Big]\, ,
\end{equation} 
where $\nu(\a)=1-|\a|(\omega_1\omega_2)^{-1/2}>0$.

Let us introduce the bilinear form on ${\mathcal H}$
$$
(U|\widehat U)=\langle A_1^{1/2}u,A_1^{1/2}\widehat u\rangle+
\langle p,\widehat p\rangle+
\langle A_2^{1/2}v,A_2^{1/2}\widehat v\rangle+
\langle q,\widehat q\rangle
 + \a \langle u,\widehat v\rangle + \a \langle v,\widehat u\rangle\,.
$$
Since
\begin{equation*}
(U|U) = 2{\mathcal E}(U)\quad \forall U\in {\mathcal H} \, ,
\end{equation*}
 thanks to \eqref{2.6z} the above  form is a scalar product on
${\mathcal H}$, and ${\mathcal H}$ is a Hilbert space with such a product.

Let now ${\mathcal A}:D({\mathcal A})\subset\hilf\to 
\hilf$ be the operator defined by
$$
\begin{cases}
D({\mathcal A})=D(A_1)\times D(A_1^{1/2})\times D(A_2)\times D(A_2^{1/2})
&
\\
{\mathcal A} U=
\left(~p\,,\,-A_1u-Bp-\a v\,,\,q\,,-A_2v-\a u~\right)
\qquad\forall
U\in D({\mathcal A})
\,.
\end{cases}
$$
Then, problem (\ref{eq:1z}) takes the equivalent form
\begin{equation}\label{eq:2z}
\begin{cases}
U'(t)={\mathcal A} U(t)
&
\\
U(0)=U_0:=(u^0,u^1,v^0,v^1)\,.
\end{cases}
\end{equation}
As will be proved in Lemma~\ref{leminvtdiopa}, ${\mathcal A}$ is a maximal dissipative operator. Then, from classical results (see, for instance, \cite{Pa}), it follows that ${\mathcal A}$ gene\-ra\-tes a 
$C_0$-semigroup, $e^{t{\mathcal A} }$, on ${\mathcal H}$. Also,
$$ e^{t{\mathcal A} }U_0 =(u(t),p(t),v(t),q(t))\,,$$
where $(u,v)$ is the solution of problem (\ref{eq:1z})-(\ref{eq:icz}), and $(p,q)=(u',v')$. 

In order to introduce our asymptotic analysis of system \eqref{eq:1z}-\eqref{eq:icz}---or, equi\-va\-lently, \eqref{eq:2z}---let us observe that, as is explained in \cite{alcako02}, no exponential stability can be expected. Therefore, weaker decay rates at infinity, such as polynomial ones, are to be sought for. Polynomial decay results for \eqref{eq:1z} were obtained in \cite{alcako02} assuming that, for some integer $j\ge 2$,
\begin{equation}\label{theta2a}
|A_1u|\le c|A_2^{j/2}u|
\qquad \forall u\in D(A_2^{j/2})\,.
\end{equation} 
Similar decay estimates for the case of boundary damping  (that is, when operator $B$ is unbounded) were derived in \cite{A2}. Also, we refer the reader to \cite{Be1}, \cite{Be2} and \cite{Youssef} for indirect stabilization with localized damping, and to \cite{BA} for the study of  a one-dimensional wave system coupled through velocities. 

The asymptotic behavior of wave-like equations and, in particular, the derivation of optimal decay rates for the energy when the geometry of the domain and damping region allow rays to be trapped, have been intensively studied for several decades. For such questions and results, we refer the reader to Lebeau~\cite{lebeauamorties} and Burq~\cite{burq} (and the references therein). In~\cite{lebeauamorties}, Lebeau considered a locally damped wave equation and proved optimal logarithmic decay rates for the energy, provided that  damping is active on a nonempty open set. The proof relies on optimal resolvent estimates for the corresponding infinitesimal generator of the associated semigroup. Later on these results were completed by Burq in~\cite{burq} in exterior domains, in particular for cases in which rays may be trapped  by the obstacle.

Independently, indirect stabilization for symmetric hyperbolic systems was first considered by the first author in~\cite{A}, and further developed in~\cite{A2,alcako02, alleau11}, using energy type methods, together with some new ideas such as the new integral inequality given in Theorem~\ref{Theorem1.1} (see~\cite{A, alcako02}). In this approach, the purpose is rather to focus on the properties of the data---that is, the operators $A_1, A_2, B$ and the coupling operator---that allow to transfer the damping action of the feedback to the undamped equation.

Subsequently, 
indirect stabilization of coupled systems was investigated in  \cite{BEPS} and \cite{LR}. In \cite{BEPS},  resolvent estimates were obtained and  spectral analysis was used to prove polynomial decay for \eqref{eq:1z}, covering some of the examples treated in \cite{alcako02}. In~\cite{LR}, where a Riesz basis approach is followed,   polynomial decay rates for the energy were derived for a simplified case of coupled system, where operators $A_1$ and $A_2$ are supposed to be equal (to $A$) and the damping operator is a nonpositive fractional power of $A$.

More recently,  inspired by \cite{lebeauamorties}
and  \cite{burq}, and, through \cite{BEPS}, by \cite{A,A2, alcako02}, the optimality of spectral-analysis-derived decay rates was shown in \cite{BD} and \cite{BT}, taking into account the asymptotic behaviour of the resolvent on the imaginary axis. 

In the context of indirect stabilization for coupled systems, we would like to stress the fact that checking  the assumptions on the data---$A_1,A_2,B$ and the coupling operator---that are needed to ensure decay, may be a difficult task. In particular, resolvent estimates  may be hard to obtain when $A_1$ and $A_2$ do not commute, or  damping and  coupling operators do not commute with $A_1$ and $A_2$. For results in this direction
we refer the reader to~\cite{A, A2}. The case of localized or boundary damping, together with localized coupling, is analyzed
in \cite{alleau11}, where $A_1=A_2=A$, but $B$ and the coupling operator do not commute with $A$. Moreover, since  coupling is localized, the corresponding operator is no longer coercive. This fact  generates additional difficulties.

In this paper, we will replace \eqref{theta2a} by
\begin{equation}\label{eq:thetab}
D(A_2)\subset D(A_1^{1/2})\qquad\mbox{and}\qquad|A_1^{1/2}u|\le c|A_2u|
\quad \forall u\in D(A_2),
\end{equation}
which is satisfied by a large class of systems  including \eqref{IS1} as a special case (see Section~\ref{se:applications} below). Under such a condition we will show that any solution $U$ of \eqref{eq:2z} satisfies the integral inequality
\begin{equation}\label{eq:main0z}
\int_{0}^{T}{\mathcal E}(U(t))dt
\leq c_1 \sum_{k=0}^{4}{\mathcal E}(U^{(k)}(0))
\qquad \forall\, T>0\, ,\ U_0\in D(\mathcal{A}^4)\, .
\end{equation}
Moreover, since the energy of solutions is decreasing in time, 
\eqref{eq:main0z} implies, in turn, the polynomial decay of order $n$ of ${\mathcal E}$, that is,
\begin{equation}\label{eq:maindz}
{\mathcal E}(U(t))
\le \frac{c_n}{t^n}\sum_{k=0}^{4n}{\mathcal E}(U^{(k)}(0))
\qquad \forall t>0
\end{equation}
for all $n\ge 1$ and $U_0\in D(\mathcal{A}^{4n})$ (see Corollary~\ref{co:main1} below). Notice that \eqref{eq:maindz} yields, in particular, the strong stability of $e^{t\mathcal{A}}$.

The compatibility condition~\eqref{eq:thetab} is equivalent to the boundedness of $A_1^{1/2}A_2^{-1}$. Let us point out that this hypothesis is sufficient but not necessary. Such a fact can be observed taking, for example,  $A_2=A_1^{\tau}$ with $\tau \in (0,1/2)$. In this case,  condition \eqref{eq:thetab} is violated, but it is easy to check that  condition~\eqref{theta2a} holds for the smallest integer $j$ such that $j > 2/\tau$. On the other hand, condition \eqref{eq:thetab} is satisfied for all $\tau \ge 1/2$.
This example shows that the present results and those of~\cite{alcako02} are in some sense complementary ---and, for $A_2=A_1^{\tau}$,($\tau \geq 0$) exactly complementary. 
One should also note that, for general operators
$A_1$ and $A_2$, the two compatibility conditions \eqref{theta2a}  and  \eqref{eq:thetab} do not cover all  possible cases.

Passing from polynomial to a general power-like decay estimate is quite natural, once \eqref{eq:maindz} has been established.  Indeed, in Section \ref{se:intres}, using {\em interpolation theory}, we obtain the fractional decay rate
\begin{equation}\label{eq:maind*z}
{\mathcal E}(U(t))
\le \frac{c_n}{t^{n/4}}\sum_{k=0}^{n}{\mathcal E}(U^{(k)}(0))
\qquad \forall t>0
\end{equation}
for all $n\ge 1$ and $U_0\in D(\mathcal{A}^n)$ (see Corollary \ref{th:main*} below). Moreover, taking initial data in $\big(\mathcal{H},D(\mathcal{A}^n)\big)_{\theta,2}$ for any $0 < \theta < 1$, 
we deduce the continuous decay rate
\begin{equation}\label{eq:maind2*z}
\|U(t)\|_{\mathcal{H}}^2
\le \frac{c_{n,\theta}}{t^{n\theta/4}}\|U_0\|_{(\mathcal{H},D(\mathcal{A}^n))_{\theta,2}}^2
\qquad \forall t>0\, .
\end{equation}
Notice that a somewhat comparable result is obtained in~\cite[Proposition~3.1] {BEPS} using a different technique. 
 In particular, for $n=1$, \eqref{eq:maind*z} implies that, for every
$
U_0\in D({\mathcal A})\, ,
$
the solution $U$ of problem (\ref{eq:2z}) satisfies 
\begin{equation}\label{eq:maind3z}
E_1(u(t),u'(t)) + E_2(v(t),v'(t))
\le \frac{c}{t^{1/4}} \|U_0\|_{D({\mathcal A})}^2\qquad \forall t>0\, ,
\end{equation}
and there exists $c_1 > 0$ such that
$$
\|U_0\|_{D({\mathcal A})}^2 \le c_1\left(|A_1 u^0|^2 + |A_1^{1/2}u^1|^2 + |A_2 v^0|^2 + |A_2^{1/2}v^1|^2\right)\, .
$$

Thus, interpolation theory applied to systems satisfying \eqref{eq:thetab} allows  to prove continuous energy decay rates, together with decay rates under explicit smoothness conditions on the initial data. Furthermore, we would like to point out that it also yields stronger results in the framework studied in \cite{alcako02}, that is, under  condition \eqref{theta2a}. We describe such applications in Section \ref{impr}, where we show how to deduce power-like decay rates from the energy estimates of \cite{alcako02}, thus recovering, in a more general set-up, related asymptotic estimates that can be obtained by spectral analysis.

Let us now mention some open questions.
One interesting problem is to derive optimal  decay rates for the energy of an indirectly damped coupled system in geometric situations for which trapped rays may exist for the uncoupled damped equation. More precisely, it would be very interesting to generalize Lebeau's resolvent analysis in \cite{lebeauamorties} to such coupled systems obtaining optimal energy estimates. In a somewhat different spirit, another open question would be to determine if it is possible to  combine the results of \cite{lebeauamorties} and \cite{burq} with the techniques developed in~\cite{A2, alcako02, alleau11} in order to derive sharp upper decay rates for the energy. In all the examples we discuss in the present work---as well as in  \cite{A, A2, alcako02, alleau11}---operators $A_1$ and $A_2$ happen  to have compact resolvents. It would be very interesting to see if explicit energy decay rates can be derived in different situations. For instance, it would be nice to extend Burq's approach~\cite{burq} in order to obtain indirect damping of coupled systems in exterior domains, and prove decay of the local total energy of  solutions.

This paper is organized as follows. Section \ref{se:preliminaries} recalls preliminary notions, mainly related to interpolation theory which is so relevant for most of  this paper. Section \ref{se:main} is devoted to our polynomial decay result and its proof. In Section \ref{se:intres}, we complete the analysis with estimates in interpolation spaces. In Section \ref{se:applications}, we describe several applications to systems of partial differential operators. Finally, in Section \ref{impr}, we show how to improve the results of \cite{alcako02} by interpolation.

\section{Preliminaries}\label{se:preliminaries}

In this section, we introduce the main tools required to deal with interpolation theory between Banach spaces. For a general exposition of this theory the reader is referred to \cite{Trie} and \cite{Lu2}. Interesting introductions are also given in \cite{BPDM} from the point of view of control theory, and \cite{Lu} for the specific case of analytic semigroups.

In this section $(X,|\,\cdot\, |_X)$ stands for a real Banach space. Let $(Y,|\,\cdot\, |_Y)$ be another Banach space. We say that $Y$ is continuously embedded into $X$, and we write 
$Y \hookrightarrow X$, if $Y\subset X$ and
\begin{equation*}
|x|_X\le c|x|_Y\qquad\forall x\in Y
\end{equation*}
for some constant $c> 0$.

We denote by $\mathcal{L}(Y;X)$ the Banach space of all bounded linear operators $T:Y\to X$ equipped with the standard operator norm. If $Y=X$, we refer to such a space as $\mathcal{L}(X)$. For any given subspace $D$ of $X$, we denote by $T_{|D}$ the restriction of $T$ to $D$.
\begin{defn}
Let $(D,|\,\cdot\, |_D)$ be a closed subspace of $X$. A subspace $(Y,|\,\cdot\, |_Y)$ of $X$ is said to be an interpolation space between $D$ and $X$ if
\begin{itemize}
\item[(a)] $D\hookrightarrow Y \hookrightarrow X$, and
\item[(b)] for every $T\in\mathcal{L}(X)$ such that $T_{|D}\in\mathcal{L}(D)$, we have that $T_{|Y}\in\mathcal{L}(Y)$.
\end{itemize}
\end{defn}
Let $X$, $D$ be Banach spaces, with $D$ continuously embedded into $X$. For any $\a\in [0,1]$, we denote by $J_\a(X,D)$ the family of all subspaces $Y$ of $X$
containing $D$ such that  
$$
|x|_Y\le c|x|_D^\a \,|x|_X^{1 - \a} \quad\forall x\in D
$$
for some constant $c>0$. 

Let us introduce, for each $x\in X$ and $t>0$, the quantity
\begin{equation}
K(t,x,X,D) := \inf_{x = a+b, \atop a\in X,\, b\in D}(|a|_X + t|b|_D)\, .
\end{equation}
Let $0 < \theta < 1$ be fixed. We define
\begin{equation}
(X,D)_{\theta,2} := \left\{x\in X : \int_0^{+\infty}|t^{-\theta - \frac{1}{2}} K(t,x,X,D)|^2\, dt < +\infty\right\}
\end{equation}
and
$$
|x|^2_{\theta,2} := \int_0^{+\infty}|t^{-\theta - \frac{1}{2}} K(t,x,X,D)|^2\, dt\, .
$$
The space $(X,D)_{\theta,2}$, endowed with the norm $|\,\cdot\,|_{\theta,2}$, is a Banach space.
\indent The reader is referred to \cite{Lu2} for the proof of the following results.

\begin{thm}\label{prop:intreswes}
Let $X_1$, $X_2$, $D_1$, $D_2$ be Banach spaces such that $D_i$ is continuously embedded in $X_i$, for $i = 1,\, 2$. If $T\in \mathcal{L}(X_1;X_2)\cap \mathcal{L}(D_1;D_2)$, then we have that $T\in\mathcal{L}((X_1,D_1)_{\theta,2};(X_2,D_2)_{\theta,2})$ for every $\theta\in (0,1)$. Moreover,
$$
\|T\|_{\mathcal{L}((X_1,D_1)_{\theta,2};(X_2,D_2)_{\theta,2})}\le \|T\|_{\mathcal{L}(X_1;X_2)}^{1 - \theta} \,\|T\|_{\mathcal{L}(D_1;D_2)}^\theta\, .
$$
\end{thm}
Consequently, the space $(X,D)_{\theta,2}$ belongs to $J_\theta(X,D)$ for every $\theta\in (0,1)$.
Let $\a\in [0,1]$ and denote by $K_\a(X,D)$ the family of all subspaces $(Y, |\,\cdot\,|_Y)$ of $X$ containing $D$ such that
$$
\sup_{t>0,\, x\in Y}\frac{K(t,x,X,D)}{t^\a |x|_Y} < +\infty \, .
$$

\begin{thm}[Reiteration Theorem]\label{th:reiteration}
Let $0 < \theta_0 < \theta_1 < 1$. Fix $\theta\in\,\,]0,1[$ and set $\o = (1-\theta)\theta_0 + \theta\theta_1$.
\begin{itemize}
\item[1)] If $E_i\in K_{\theta_i}(X,D)$, $i = 0,\,1$, then $(E_0,E_1)_{\theta,2}\subset (X,D)_{\o,2}\, .$
\item[2)] If $E_i\in J_{\theta_i}(X,D)$, $i = 0,\,1$, then $(X,D)_{\o,2}\subset (E_0,E_1)_{\theta,2}$.
\end{itemize}
Consequently, if $E_i\in J_{\theta_i}(X,D)\cap K_{\theta_i}(X,D)$, $i = 0,\,1$, then $(E_0,E_1)_{\theta,2} = (X,D)_{\o,2}$, with equivalence between the respective norms.
\end{thm}

\begin{rem}\label{remreiterationth}
Since $(X,D)_{\theta,2}$ is contained in $J_{\theta}(X,D)\cap K_{\theta}(X,D)$, for every $0<\theta_0,\theta_1<1$ we have
\begin{eqnarray}
\left((X,D)_{\theta_0,2},(X,D)_{\theta_1,2}\right)_{\theta,2} = (X,D)_{(1-\theta)\theta_0 + \theta\theta_1,2}\, .
\end{eqnarray}
Since $X\in J_0(X,D)\cap K_0(X,D)$ and $D\in J_1(X,D)\cap K_1(X,D)$, we also have
\begin{eqnarray}
\left(X,(X,D)_{\theta_1,2}\right)_{\theta,2} = (X,D)_{\theta\theta_1,2}\quad \textrm{and} \\
\left((X,D)_{\theta_0,2},D\right)_{\theta,2} = (X,D)_{(1-\theta)\theta_0 + \theta,2}\, .
\end{eqnarray}
\end{rem}

\subsection{Interpolation spaces and fractional powers of operators}

Let $(H,\langle\,\cdot\, ,\,\cdot\,\rangle)$ be a real Hilbert space, with norm $|\,\cdot\,|$. Let $A:D(A)\subset H\to H$ be a densely defined closed linear operator on $H$ such that
\begin{equation}\label{eq:positivity}
\langle Ax,x\rangle \ge \delta |x|^2\, ,\quad \forall\, x\in D(A)
\end{equation}
for some $\delta > 0$. As usual, we denote by $A^\theta$ the fractional power of $A$ for any $\theta\in\mathbb{R}$ (see, for instance, \cite[Chapter 1 - Section 5]{BPDM}), and  by $A^*$ the adjoint of $A$.
We recall that $A$ is self-adjoint if $D(A) = D(A^*)$ and $\langle Ax,y\rangle = \langle x,Ay\rangle$ for every $x,\, y\in D(A)$.
For the proof of the following result we refer to \cite[Theorem 4.36]{Lu2}. 

\begin{thm}\label{thmintspfractpow}
Let $A$ be a self-adjoint operator satisfying \eqref{eq:positivity}. Then, for every $\theta\in (0,1)$, $\alpha,\, \beta\in\mathbb{R}$ such that $\beta> \alpha\ge 0$,
\begin{equation}\label{eq:interp2}
(D(A^\alpha),D(A^\beta))_{\theta,2} = D(A^{(1-\theta)\alpha + \theta\beta})\, .
\end{equation}
In particular,
\begin{equation}\label{eq:interp1}
(H,D(A^\beta))_{\theta,2} = D(A^{\beta\theta})\, .
\end{equation}
\end{thm}

We say that $A$ is an m-accretive operator if

$$
\begin{cases}
\langle Ax,x\rangle \ge 0 & \forall x\in D(A) \quad (accretivity) \\
(\lambda I + A)D(A) = H & \text{for some } \lambda > 0 \quad (maximality)
\end{cases}
$$
Notice that, if the above maximality condition is satisfies for some $\lambda > 0$, then the same condition holds for every $\lambda > 0$. Moreover, we say that $A$ is m-dissipative if $-A$ is m-accretive.

We refer the reader to \cite[Section 4.3]{Lu2} for the proof of the next result.

\begin{prop}\label{propfractpow}
Let $(A,D(A))$ be an m-accretive operator on a Hilbert space $H$, with $A^{-1}$ bounded in $H$.
Then for every $\alpha,\, \beta\in\mathbb{R}$, $\beta> \alpha\ge 0$, $\theta\in (0,1)$, $A$ satisfies \eqref{eq:interp2} and \eqref{eq:interp1}. In particular,
\begin{equation}\label{eq:propfractpow}
D(A^\theta) = (H,D(A))_{\theta,2}\quad \forall\, 0 < \theta < 1\, .
\end{equation}
\end{prop}

\begin{cor}\label{rem:interp2}
If $A$ is the infinitesimal generator of a $\mathcal{C}_0$-semigroup of contractions on $H$, with $A^{-1}$ bounded in $H$, then $D(A^m) = (H,D(A^k))_{\theta,2}$ for every $k\in \mathbb{N}$, $\theta\in (0,1)$ such that $m = \theta k$ is an integer.
\end{cor}

\subsection{An abstract decay result}\label{se:abstract}

We recall an abstract result obtained in \cite{A} in a slightly different form, and in \cite[Theorem~2.1]{alcako02} in the current version.\\
Let $A:D(A)\subset H\to H$ be the infinitesimal generator of a $\mathcal{C}_0$-semigroup of bounded linear operators on $H$.

\begin{thm}\label{Theorem1.1}
Let $L:H\to [0,+\infty)$ 
be a continuous function such that, for some integer $K\ge 0$ and some constant $c\ge 0$,
\begin{equation} \int_{0}^{T}
L (e^{tA}x)dt\le c \sum_{k=0}^{K}L (A^kx)\qquad\forall\, T\ge 0\,,\;\forall\, x\in D(A^K)\,.
\label{1.4}
\end{equation}
Then, for any integer $n\ge 1$, any $x\in D(A^{nK})$ and any $0\le s\le T$
\begin{equation}\label{1.5}
\int_{s}^{T}L (e^{tA}x) \;\frac{(t-s)^{n-1}}{(n-1)!}\,dt
 \leq c^n(1+K)^{n-1}\sum_{k=0}^{nK} 
 L (e^{sA}A^kx)\, .
\end{equation}
If, in addition, $L(e^{tA}x)\le L(e^{sA}x)$ for any $x\in H$ and any $0\le s\le t$, then
\begin{equation}\label{1.6}
L (e^{tA}x) 
 \leq  c^n (1+K)^{n-1}\,\frac{n!}{t^{n}}
\sum_{k=0}^{nK} 
 L (A^k x)\qquad\forall t>0
\end{equation}
for any integer $n\ge 1$ and any $x\in D(A^{nK})$.
\end{thm}

\section{Main result}\label{se:main}

We are now ready to state and prove the polynomial decay of solutions to weakly coupled systems. In addition to the standing assumptions $(H1),(H2),(H3)$, we will assume that
\begin{equation}\label{eq:theta}
D(A_2)\subset D(A_1^{1/2})\qquad\mbox{and}\qquad|A_1^{1/2}u|\le c|A_2u|
\quad \forall u\in D(A_2)
\end{equation}
for some constant $c>0$. Condition \eqref{eq:theta} can be formulated in the following equivalent ways.

\begin{lem}\label{le:main}
Under assumption $(H1)$ the following properties are equivalent.
\begin{enumerate}
\item[(a)]
Assumption \eqref{eq:theta} holds.
\item[(b)] $A_1^{1/2}A_2^{-1}\in \mathcal L(H)$.
\item[(c)] For some constant $c>0$
\begin{equation}\label{eqnequiv}
|\langle A_1u,v\rangle|\le c|A_2v|\langle A_1u,u\rangle^{1/2}\qquad\forall u\in D(A_1)\,,\;\forall v\in D(A_2)\,.
\end{equation}
\end{enumerate}
\end{lem}

\begin{proof}
The implications (a)$\Leftrightarrow$(b)$\Rightarrow$(c) being straightforward, let us proceed to show that (c)$\Rightarrow$(a). Consider the Hilbert space
$V_1=D(A_1^{1/2})$ with the scalar product 
\begin{equation*}
\langle u,v\rangle_{V_1} =\langle A_1^{1/2}u,A_1^{1/2}v\rangle
\end{equation*}
and recall that $D(A_1)$ is a dense subspace of $V_1$. Let $v\in D(A_2)$ and define the linear functional $\phi_v:D(A_1)\to\R$ by
\begin{equation*}
\phi_v(u)=\langle A_1u,v\rangle\qquad\forall u\in D(A_1)\,.
\end{equation*}
Owing to (c), $\phi_v$ can be extended to a bounded linear functional on $V_1$ (still denoted by $\phi_v$) satisfying $\|\phi_v\|\leq c|A_2v|$. Therefore, by the Riesz Theorem, there is a unique vector $w\in V_1$ such that
\begin{equation*}
\phi_v(u)=\langle A_1^{1/2}u,A_1^{1/2}w\rangle
\qquad\forall u\in V_1\,.
\end{equation*}
Hence, $\langle A_1u,(v-w)\rangle=0$ for all $u\in D(A_1)$, and so $v=w\in V_1$ since $A_1$ is invertible. Moreover, $|A_1^{1/2}v|=|w|_{V_1}\leq c|A_2v|$.
\end{proof}

The main result of this section is the following.

\begin{thm}\label{th:main}
Assume $(H1),(H2),(H3)$ and \eqref{eq:theta}. 
If $U_0\in D({\mathcal A}^{4})$, then the solution $U$ of problem \eqref{eq:2z}
satisfies 
\begin{equation}\label{eq:main0}
\int_{0}^{T}{\mathcal E}(U(t))dt
\leq c_1 \sum_{k=0}^{4}{\mathcal E}(U^{(k)}(0))
\qquad \forall\, T>0
\end{equation}
for some constant $c_1>0$.
\end{thm}

The proof of Theorem \ref{th:main} will be given in several steps. First, let us recall that, as showed in \cite[Lemma~3.3]{alcako02}, system \eqref{eq:2z} is dissipative. Indeed, under the only assumptions $(H1)$ and $(H2)$, the energy of the solution $U =(u,u',v,v')$ of problem \eqref{eq:2z} with $U_0\in D({\mathcal A})$ satisfies
\begin{equation}\label{eq:dissi}
\frac d{dt}{\mathcal E}(U(t))
=-|B^{1/2}u'(t)|^2
\qquad \forall t\geq 0.
\end{equation}
In particular, $t\mapsto {\mathcal E}(U(t))$ is nonincreasing on $[0,\infty)$.

\begin{cor}\label{co:main1}
Assume $(H1),(H2),(H3)$ and \eqref{eq:theta}.
\begin{itemize}
\item[(a)] If $U_0\in D({\mathcal A}^{4n})$ for some integer
$n\ge 1$, then the solution $U$ of problem \eqref{eq:2z}
satisfies 
\begin{equation}\label{eq:maind}
{\mathcal E}(U(t))
\le \frac{c_n}{t^n}\sum_{k=0}^{4n}{\mathcal E}(U^{(k)}(0))
\qquad \forall t>0
\end{equation}
for some constant $c_n>0$.
\item[(b)]  For every $U_0\in {\mathcal H}$ we have
$$
{\mathcal E}(U(t))\to 0\quad\hbox{as}\quad t\to +\infty.
$$ 
\end{itemize}
\end{cor}
\begin{proof}
Statement $(a)$ follows by combining the dissipation relation \eqref{eq:dissi}, Theorem~\ref{th:main}, and Theorem~\ref{Theorem1.1}. To prove part $(b)$, we fix $U_0\in \h$ and consider a sequence $(U_0^n)_{n\in\mathbb{N}}$ such that $U_0^n\in D(\A^{4n})$ for every $n\ge 1$ and $U_0^n\to U_0$ in $\h$ for $n\to +\infty$. We set
$U^n(t)=e^{t{\mathcal A}}U^n_0$ and $U(t)=e^{t{\mathcal A}}U_0$ for $t \ge 0$. Then, by linearity and the contraction property of  $(e^{t{\mathcal A}})_{t \ge 0}$, we have
$$
||U^n(t) - U(t)|| \le ||U_0^n - U_0|| \,, \quad\forall \ t \ge 0\,, n \in
\mathbb{N} \,.
$$
\noindent Therefore, recalling the definition of  ${\mathcal E}$, we deduce that
${\mathcal E}(U^n(.))$ converges to
${\mathcal E}(U(.))$ as $n\to +\infty$, uniformly on $[0,\infty)$. Since, for any fixed $n \in \mathbb{N}$, ${\mathcal E}(U^n(t))$ converges to $0$ as $t\to \infty$, we easily obtain the conclusion.
\end{proof}
We now proceed with the proof of Theorem~\ref{th:main}. Hereafter, $C$ will denote a generic positive constant, independent of $\a$. To begin with, let us recall that, thanks to \cite[Lemma~3.4]{alcako02}, the solution of \eqref{eq:2z} with $U_0\in D({\mathcal A})$ verifies
\begin{equation}\label{eq:decay}
\int_{0}^{T}{\mathcal E}(U(t))dt
 \leq
\int_{0}^{T}|v'(t)|^2dt +C\,{\mathcal E}(U(0))
\end{equation}
for some constant $C\ge 0$  and every $T\ge 0$.
Hence, the main technical point of the proof is to bound the right-hand side of \eqref{eq:decay} by the total energy of $U$ (and a finite number of its derivatives) at 0.

\begin{lem}
Let $U=(u,u',v,v')$ be the solution of problem \eqref{eq:2z} with $U_0\in D(\mathcal A)$. Then
\begin{equation}\label{eq:est1}
\int_0^T|A_1^{-1/2}v|^2dt\leq C\int_0^T|A_2^{-1/2}u|^2dt+ {C\over\a^2}\Big[{\mathcal E}(U(0))+{\mathcal E}(U'(0))\Big]\, .
\end{equation}
\end{lem}

\begin{proof}
Rewrite \eqref{eq:2z} as system \eqref{eq:1z} to obtain
\begin{multline*}
\int_0^T\langle u''+A_1u+Bu'+\a v,A_1^{-1}v\rangle dt
-\int_0^T\langle v''+A_2v+\a u,A_2^{-1}u\rangle dt=0\, .
\end{multline*}
Hence, by straightforward computations,
\begin{multline*}
\a\int_0^T|A_1^{-1/2}v|^2dt\leq\a\int_0^T|A_2^{-1/2}u|^2dt
\\
-\int_0^T\langle Bu',A_1^{-1}v\rangle dt+\int_0^T\big[\langle v'',A_2^{-1}u\rangle-
\langle u'',A_1^{-1}v\rangle\big]dt\,.
\end{multline*}
Integration by parts transforms the last inequality into
\begin{multline}\label{eq:lem1}
\a\int_0^T|A_1^{-1/2}v|^2dt\leq\a\int_0^T|A_2^{-1/2}u|^2dt
-\int_0^T\langle A_1^{-1/2}Bu',A_1^{-1/2}v\rangle dt
\\
+\int_0^T\big[\langle  A_1^{-1/2}v, A_1^{1/2}A_2^{-1}u''\rangle-
\langle  A_1^{-1/2}u'',A_1^{-1/2}v\rangle\big]dt
\\+
 \Big[\langle v',A_2^{-1}u\rangle-\langle v,A_2^{-1}u'\rangle\Big]_0^T\,.
\end{multline}
We now proceed to bound the right-hand side of \eqref{eq:lem1}. We have
\begin{equation*}
\left|\int_0^T\langle A_1^{-1/2}Bu',A_1^{-1/2}v\rangle dt\right|\leq
{\a\over 4}\int_0^T|A_1^{-1/2}v|^2dt+ {C\over\a} \int_0^T |B^{1/2}u'|^2dt\,.
\end{equation*}
Similarly, thanks to assumption \eqref{eq:theta} and the fact that $B$ is positive definite, 
\begin{equation*}
\left|\int_0^T\langle  A_1^{-1/2}v, A_1^{1/2}A_2^{-1}u''\rangle dt\right|\leq
{\a\over 4}\int_0^T|A_1^{-1/2}v|^2dt+ {C\over\a} \int_0^T |B^{1/2}u''|^2dt\, .
\end{equation*}
Also,
\begin{equation*}
\left|\int_0^T\langle  A_1^{-1/2}u'',A_1^{-1/2}v\rangle dt\right|\leq
{\a\over 4}\int_0^T|A_1^{-1/2}v|^2dt+ {C\over\a} \int_0^T |B^{1/2}u''|^2dt\,.
\end{equation*}
Finally, observe that the last term in \eqref{eq:lem1} can be bounded as follows
\begin{equation*}
\left| \Big[\langle v',A_2^{-1}u\rangle-\langle v,A_2^{-1}u'\rangle\Big]_0^T\right|
\leq C{\mathcal E}(U(0))\,.
\end{equation*}
Combining the above estimates with \eqref{eq:lem1}, we obtain
\begin{multline*}
\int_0^T|A_1^{-1/2}v|^2dt\leq C\int_0^T|A_2^{-1/2}u|^2dt
+{C\over\a}{\mathcal E}(U(0))
\\
+{C\over\a^2} \int_0^T \big[|B^{1/2}u'|^2+|B^{1/2}u''|^2\big]dt\, .
\end{multline*}
The conclusion follows from the above inequality and the dissipation identity \eqref{eq:dissi}
applied to $u$ and $u'$.
\end{proof}

\begin{lem}
Let $U=(u,u',v,v')$ be the solution of problem \eqref{eq:2z} with $U_0\in D(\mathcal A)$. Then
\begin{equation}\label{eq:est2}
\int_0^T|v|^2dt\leq C\a^2\int_0^T|u|^2dt+ {C\over\a^2}\sum_{k=1}^3{\mathcal E}(U^{(k)}(0))\, .
\end{equation}
\end{lem}
\begin{proof}
Since $\langle v''+A_2v+\a u,A_2^{-1}v\rangle=0$, integrating over $[0,T]$ we have
\begin{equation}\label{eq:2.4}
\int_0^T|v|^2dt=-\a \int_0^T\langle v, A_2^{-1}u\rangle dt
-\int_0^T\langle  v'', A_2^{-1}v\rangle dt\,.
\end{equation}
The last term in the above identity can be bounded using assumption \eqref{eq:theta} and Lemma~\ref{le:main} as follows
\begin{multline}\label{eq:2.5}
\left|\int_0^T\langle  v'', A_2^{-1}v\rangle dt\right|
=\left|\int_0^T\langle  A_1^{-1/2}v'', A_1^{1/2}A_2^{-1}v\rangle dt\right|
\\
\le {1\over 4}\int_0^T|v|^2dt
+C\int_0^T|A_1^{-1/2}v''|^2dt\, .
\end{multline}
Now, applying \eqref{eq:est1} to $v''$ and \eqref{eq:dissi} to $u'$, we obtain
\begin{multline}\label{eq:2.6}
\int_0^T|A_1^{-1/2}v''|^2dt\leq C \int_0^T|A_1^{-1/2}u''|^2dt+
{C\over\a^2}\big[{\mathcal E}(U''(0))+{\mathcal E}(U'''(0))\big]
\\
\leq C{\mathcal E}(U'(0))+ {C\over\a^2}\big[{\mathcal E}(U''(0))+{\mathcal E}(U'''(0))\big]\, .
\end{multline}
On the other hand,
\begin{equation}\label{eq:2.7}
\left|\a \int_0^T\langle v, A_2^{-1}u\rangle dt\right|
\leq {1\over 4}\int_0^T|v|^2dt + C\a^2\int_0^T|u|^2dt\, .
\end{equation}
The conclusion follows combining \eqref{eq:2.4},\dots,\eqref{eq:2.7}. 
\end{proof}
Let us now complete the proof of Theorem~\ref{th:main}.

\begin{proof}[Proof of Theorem~\ref{th:main}]
To prove \eqref{eq:main0} it suffices to apply \eqref{eq:est2} to $v'$ and use the resulting estimate to bound the right-hand side of \eqref{eq:decay}.  Since $B$ is positive definite, the conclusion follows by the dissipation identity \eqref{eq:dissi}.
\end{proof}

\begin{rem}
$(i)$ Similar results can be obtained for systems of equations coupled with different coefficients, such as
\begin{equation}\label{eq:temp}
\begin{cases}
u''(t)+A_1u(t)+Bu'(t)+\a_1 v(t)=0
\\
v''(t)+A_2v(t)+\a_2 u(t)=0 \, .
\end{cases}
\end{equation}
In this case, $(H3)$ should be replaced with

\begin{enumerate}
\item[(H3')] $\a_1$, $\a_2$ are two real numbers such that $0< \a_1\a_2 < \o_1\o_2$.
\end{enumerate}
Let us explain how to adapt our approach to the case of $\a_1 \neq \a_2$, when $\a_1$, $\a_2 > 0$. The total energy is defined by
$$
\mathcal{E}(U) := \a_2 E_1(u,p) + \a_1 E_2(v,q) + \a_1 \a_2 \langle u,v \rangle\, ,
$$
where $E_1$ and $E_2$ are the energies of the two components, defined in \eqref{eq:enercompsz}. Moreover, for each $U\in\mathcal{H}$,
$$
\mathcal{E}(U) \ge \nu(\a_1,\a_2) \Big[\a_2 E_1(u,p) + \a_1 E_2(v,q)\Big]\, ,
$$
with $\nu(\a_1,\a_2) = 1 - (\a_1 \a_2)^{1/2}(\o_1 \o_2)^{-1/2} > 0$. Finally, for each $U_0\in D(\mathcal{A})$, the solution $U(t) = (u(t), p(t),v(t),q(t))$ of the first order evolution equation associated with system \eqref{eq:temp} satisfies
\begin{equation}
\frac d{dt}{\mathcal E}(U(t))
=-\a_2|B^{1/2}u'(t)|^2
\qquad \forall t\geq 0.
\end{equation}
In particular, $t\mapsto {\mathcal E}(U(t))$ is nonincreasing on $[0,\infty)$. From this point, reasoning as in the above proof, the reader can easily derive the conclusion of Theorem \ref{th:main}.

$(ii)$ Another interesting situation occurs when $\a_1 = 0$, that is, when the first equation of system \eqref{eq:temp} is damped, whereas the second component is undamped and weakly coupled with the first one. In this case there is no hope to stabilize the full system by one single feedback. Indeed, let $A_1 = A_2 =: A$  and consider the sequence of positive eigenvalues $(\o_k)_{k\ge 1}$ of $A$, satisfying $\o_k\to +\infty$, with associated eigenspaces $(Z_k)_{k\ge 1}$. Moreover, let $B = 2\beta I$, with $0 < \beta < \sqrt{\o_1}$, and $\lambda_k = \sqrt{\o_k - \beta^2}$. Then, the equation
\begin{equation}\label{eqrmku}
u''(t)+Au(t)+2\beta u'(t) =0
\end{equation}
with initial conditions
$$
u(0) = u^0 = \sum_{k\ge 1}u^0_k\, , \quad u'(0) = u^1 = \sum_{k\ge 1} u^1_k\, ,
$$
where $u^i_k\in Z_k$ for every $k\ge 1$, ($i= 1,\,2$), admits the solution
\begin{equation*}
u(t) = e^{-\beta t}\sum_{k\ge 1} \left[u^0_k \cos(\lambda_k t) + \frac{u^1_k + \beta u^0_k}{\lambda_k}\sin(\lambda_k t)\right]\, .
\end{equation*}
In particular, choosing $u^0\in Z_1$ and $u^1 \in Z_1$, we have that $u(t)$ lies in $Z_1$ for every $t\ge 0$.
On the other hand, the solution to
\begin{equation}\label{eqtempsuv}
v''(t)+Av(t)+\a u(t)=0
\end{equation}
is coupled with \eqref{eqrmku} only in the component in $Z_1$, while it is conservative in $Z_1^\perp$. More precisely, writing $v(t) = v_1(t) + v_2(t)\in Z_1 + Z_1^\perp$, equation \eqref{eqtempsuv} implies that
\begin{equation}\label{eq:temprmk}
\begin{cases}
v_1''(t)+ \o_1 v_1(t)+\a u(t)=0
\\
v_2''(t)+Av_2(t) = 0\, .
\end{cases}
\end{equation}
Therefore, taking $v(0) = v^0\notin Z_1$ and $v'(0) = v^1 \notin Z_1$, 
$$
E(v_2(t),v_2'(t)) = \frac{1}{2}\left(|v_2'(t)|^2 + \langle A v_2(t),v_2(t)\rangle \right) = E(v(0),v'(0))>0$$
for all $t\ge 0$.
So,  system \eqref{eq:temp} is not stabilizable.
\end{rem}

\section{Results with data in interpolation spaces}\label{se:intres}

When the initial data belong to an interpolation space between $\mathcal{H}$ and the domain of a power of $\mathcal{A}$ we can improve Corollary \ref{co:main1} as follows.

\begin{thm}\label{thm:intresult}
Assume $(H1),(H2),(H3)$ and \eqref{eq:theta}. If $U_0\in (\mathcal{H},D(\mathcal{A}^{4n}))_{\theta,2}$ for some $n\ge 1$ and $0<\theta < 1$, then the solution $U$ of problem \eqref{eq:2z}
satisfies 
\begin{equation}\label{eq:maind2}
\|U(t)\|_{\mathcal{H}}^2
\le \frac{c_{n,\theta}}{t^{n\theta}}\|U_0\|_{(\mathcal{H},D(\mathcal{A}^{4n}))_{\theta,2}}^2
\qquad \forall t>0
\end{equation}
for some constant $c_{n,\theta}>0$.
\end{thm}
\begin{proof}
The proof easily follows from the interpolation results recalled in Section 2 applied to the operator $\Lambda_t:\mathcal{H}\to \mathcal{H}$ defined by
\begin{displaymath}
\Lambda_t(U_0) = e^{t{\mathcal A}}U_0 \in\mathcal{H}
\end{displaymath}
for each $U_0\in \mathcal{H}$.
\end{proof}
Although $(\mathcal{H},D(\mathcal{A}^{4n}))_{\theta,2}$ is usually difficult to identify explicitly, we can single out important special cases where such an identification is possible. We need a preliminary result.

\begin{lem}\label{leminvtdiopa}
The operator $\mathcal{A}: D(\mathcal{A})\to \h$ is invertible, with $\mathcal{A}^{-1}$ bounded. Moreover, $\mathcal{A}$ is m-dissipative (thus, $\mathcal{A}$ generates a $\mathcal{C}_0$-semigroup of contractions on $\h$).
\end{lem}

\begin{proof}
For any $U = (u,p,v,q)$, $\widehat{U} = (\hat{u},\hat{p},\hat{v},\hat{q})\in \mathcal{H}$, the identity $\mathcal{A}U = \widehat{U}$ is equivalent to
$$
p  = \hat{u}\, ,\quad -A_1u - Bp - \a v = \hat{p}\, ,\quad q = \hat{v}\, ,\quad -A_2 v - \a u = \hat{q}\, .
$$
Hence, $p  = \hat{u} \in D(A_1^{1/2})$, $q = \hat{v}\in D(A_2^{1/2})$. So, in order to compute $\mathcal{A}^{-1}$ it suffices to solve the system
\begin{equation}\label{eq:laxmilg}
\begin{cases}
A_1u +\a v = f
\\
A_2v +\a u = g \, ,
\end{cases}
\end{equation}
for suitably chosen $f,\, g\in H$. Since $I-\a^2A_1^{-1}A_2^{-1}$ is invertible thanks to $(H3)$, it is easy to check that \eqref{eq:laxmilg} admits the solution
\begin{equation*}
\begin{cases}
\bar{u} = \left(I-\a^2A_1^{-1}A_2^{-1}\right)^{-1}A_1^{-1}(f - \a A_2^{-1}g) \in D(A_1)
\\
\bar{v} = A_2^{-1}(g - \a \bar{u})\in D(A_2) \, .
\end{cases}
\end{equation*}
Thus, $\mathcal{A}$ is invertible, and $\mathcal{A}^{-1}$ is bounded.
Moreover, $\mathcal{A}$ is dissipative, since
$$
(\mathcal{A} U| U) \le -\langle Bp,p\rangle_H \le -\beta |p|_H^2 \le 0\quad \forall\, U\in D(\mathcal{A})\, .
$$
In addition, it is easy to check that there exists $\lambda > 0$ such that the range of $\lambda I - \mathcal{A}$ equals $\h$. Thus, by the Lumer-Phillips Theorem (see, e.g., \cite[Theorem~4.3]{Pa}), $\mathcal{A}$ generates a $\mathcal{C}_0$-semigroup of contractions on $\h$.
\end{proof}

Applying Corollary \ref{rem:interp2}, we obtain the following result.

\begin{cor}\label{lem:intpows}
If $\theta k = m$, for some $0 < \theta < 1$ and $k,\,m\in\mathbb{N}$, then
\begin{equation}\label{intpowscra}
D(\mathcal{A}^m) = (\mathcal{H},D(\mathcal{A}^k))_{\theta,2}\, .
\end{equation}
\end{cor}

\begin{rem}\label{rmkcnst}
In particular, let us take $k=4n\;(n\ge 1)$ and $\theta_j = \frac{j}{4n}$ for $j=1, \dots, 4n - 1$.  Then, \eqref{intpowscra}
yields
\begin{equation}\label{eq:wuppertal}
(\mathcal{H},D(\mathcal{A}^{4n}))_{\theta_j,2} = D(\mathcal{A}^{j})
\quad (j=1, \dots, 4n - 1)\,.
\end{equation}
Thus, applying Theorem \ref{thm:intresult} to the above values of $\theta_j$, one can show that, if $U_0\in D(\mathcal{A}^{j})$, then the associated solution $U(t)$ of problem \eqref{eq:2z} satisfies
$$
\|U(t)\|_{\mathcal{H}}^2
\le \frac{c_{n,j}}{t^{j/4}}\|U_0\|_{D(\mathcal{A}^{j})}^2
\qquad \forall t>0
$$
for some constant $c_{n,j}>0$. Moreover, we claim that $c_{n,j}$ can be chosen independent of $n$. Indeed, since $j\neq 4n$,  one can take the smallest positive $n_j$ such that $j< 4n_j$, and use \eqref{eq:wuppertal}  with $\theta_j = j/(4n_j)$ to conclude that $c_{n_j,j} = c_j$. As already mentioned in the introduction, this result can be compared with the one in~\cite[Proposition 3.1] {BEPS},  which was obtained by a different method.
\end{rem}

\begin{cor}\label{th:main*}
Assume $(H1),(H2),(H3)$ and \eqref{eq:theta}.
\begin{itemize}
\item[i)] If $U_0\in D({\mathcal A}^{n})$ for some $n\ge 1$, then the solution of \eqref{eq:2z} satisfies
\begin{equation}\label{eq:maind*}
{\mathcal E}(U(t))
\le \frac{c_n}{t^{n/4}}\sum_{k=0}^{n}{\mathcal E}(U^{(k)}(0))
\qquad \forall t>0
\end{equation}
for some constant $c_n>0$.
\item[ii)] If $U_0\in (\mathcal{H},D(\mathcal{A}^n))_{\theta,2}$ for some $n\ge 1$ where $0<\theta < 1$, then the solution of \eqref{eq:2z} satisfies
\begin{equation}\label{eq:maind2*}
\|U(t)\|_{\mathcal{H}}^2
\le \frac{c_{n,\theta}}{t^{n\theta/4}}\|U_0\|_{(\mathcal{H},D(\mathcal{A}^n))_{\theta,2}}^2
\qquad \forall t>0
\end{equation}
for some constant $c_{n,\theta}>0$.
\item[iii)] If $U_0\in D((-\mathcal{A})^\theta)$ for some $0<\theta < 1$, then the solution of problem \eqref{eq:2z}
satisfies
\begin{equation}\label{eq:interm}
\|U(t)\|_{\mathcal{H}}^2
\le \frac{c_{\theta}}{t^{\theta/4}}\|U_0\|_{D((-\mathcal{A})^\theta)}^2
\qquad \forall t>0
\end{equation}
for some constant $c_{\theta}>0$.
\end{itemize}
\end{cor}

\begin{proof}
Points $i)$ and $ii)$ derive from Corollary \ref{co:main1} and following the proof of Theorem \ref{thm:intresult}, thanks to Remark \ref{rmkcnst}. In order to prove point $iii)$, first we deduce from Lemma~\ref{leminvtdiopa} that $-\mathcal{A}$ is invertible with bounded inverse. Moreover, it is m-accretive on $\mathcal{H}$, hence \eqref{eq:propfractpow} yields
$$
(\h,D(\mathcal{A}))_{\theta,2} = (\h,D(-\mathcal{A}))_{\theta,2} = D((-\mathcal{A})^\theta)
$$
for every $0< \theta < 1$. The conclusion follows applying  $ii)$ with
$n=1$.
\end{proof}

Under further assumptions, the norm in $(\h,D(\mathcal{A}))_{\theta,2}$ can be given a more explicit form.
For this purpose, for each $k\ge 0$ consider the space
$$
\mathcal{H}_k = D(A_1^{(k+1)/2})\times D(A_1^{k/2})\times D(A_2^{(k+1)/2})\times D(A_2^{k/2})\, .
$$
We recall the following result (see \cite[Lemma~3.1]{alcako02}).

\begin{lem}\label{domain}
Assume $(H1)$, and $(H2)$. Let $n\ge 1$ be such that
\begin{eqnarray}\label{domain1}
& BD(A_1^{(k+1)/2})\subset D(A_1^{k/2})&
\\
\label{domain2}
& D(A_1^{(k/2)+1})\subset D(A_2^{k/2})&
\\
\label{domain3}
& D(A_2^{(k/2)+1})\subset D(A_1^{k/2})&
\end{eqnarray}
for every integer $k$ satisfying $0< k\le n-1$. (no assumption is made if $n=1$). Then ${\mathcal H}_k\subset
D({\mathcal A}^k)$ for every $0\le k\le n$.
\end{lem}

In \cite{alcako02}, it is also shown  that ${\mathcal H}_k = D({\mathcal A}^k)$ for every $0\le k\le n$, provided \eqref{domain2} and \eqref{domain3} are replaced by the stronger assumptions
$$
\begin{array}{c}
 D(A_1^{(k+1)/2})\subset D(A_2^{k/2})
\\
 D(A_2^{(k+1)/2})\subset D(A_1^{k/2})
\end{array}
\qquad\mbox{for every}\quad 0< k\le n-1\,.
$$

Let $0 < \theta < 1$ and $k\ge 1$ be fixed. As a direct consequence of Theorem~\ref{prop:intreswes}, choosing appropriate spaces and
operator $T$, one can show that, if ${\mathcal H}_k$ is contained in $D({\mathcal A}^k)$, then $(\mathcal{H},\mathcal{H}_k)_{\theta,2}$ is contained in $(\mathcal{H},D(\mathcal{A}^k))_{\theta,2}$. Moreover, $(\mathcal{H},\mathcal{H}_k)_{\theta,2}$ equals
$$
\begin{array}{ccl}
\mathcal{H}_{k,\theta} & := & (D(A_1^{1/2}),D(A_1^{(k+1)/2}))_{\theta,2}\times (H,D(A_1^{k/2}))_{\theta,2} \\
 &  & \times (D(A_2^{1/2}),D(A_2^{(k+1)/2}))_{\theta,2}\times (H,D(A_2^{k/2}))_{\theta,2}\, .
\end{array}
$$
Notice that, since $A_i$ is self-adjoint and \eqref{eq:positivity} holds for $i = 1,\,2$, applying Theorem \ref{thmintspfractpow} we have, for every $0\le \alpha < \beta\;(i = 1,\,2)$,
$$
(D(A_i^\alpha),D(A_i^\beta))_{\theta,2} = D(A_i^{(1-\theta)\alpha + \theta\beta})\,.
$$
Therefore, $\mathcal{H}_{k,\theta}$ equals $D(A_1^{\frac{1}{2} + \frac{k}{2}\theta}) \times D(A_1^{\frac{k}{2}\theta}) \times D(A_2^{\frac{1}{2} + \frac{k}{2}\theta}) \times D(A_2^{\frac{k}{2}\theta})\, .$

Observing that, for initial data in $\mathcal{H}_{n,\theta}$, we can bound (above and below) the norm of $U_0$ by the norms of its components, we have the following.

\begin{cor}\label{th:resintspbetbound}
Assume $(H1),(H2),(H3)$ and \eqref{eq:theta}.

$1)$ If ${\mathcal H}_n\subset D({\mathcal A}^n)$ for some $n\ge 2$, then for each $U_0\in \mathcal{H}_n$ the solution $U$ of problem \eqref{eq:2z}
satisfies
\begin{equation}
\|U(t)\|_{\mathcal{H}}^2
\le \frac{c_{n}}{t^{n/4}}\|U_0\|_{\mathcal{H}_{n}}^2
\qquad \forall t>0
\end{equation}
for some constant $c_{n}>0$, where
$$
\|U_0\|_{\mathcal{H}_{n}}^2 = |u^0|_{D(A_1^{(n+1)/2})}^2 + |u^1|_{D(A_1^{n/2})}^2 + |v^0|_{D(A_2^{(n+1)/2})}^2 + |v^1|_{D(A_2^{n/2})}^2\, .
$$

$2)$ Let $n\ge 1$ and $0<\theta < 1$ be fixed. If ${\mathcal H}_n\subset D({\mathcal A}^n)$, then for every $U_0\in \mathcal{H}_{n,\theta}$ the solution $U$ of \eqref{eq:2z}
satisfies
\begin{equation}
\|U(t)\|_{\mathcal{H}}^2
\le \frac{c_{n,\theta}}{t^{n\theta/4}}\|U_0\|_{\mathcal{H}_{n,\theta}}^2
\qquad \forall t>0
\end{equation}
for some constant $c_{n,\theta}>0$, with
$$
\|U_0\|_{\mathcal{H}_{n,\theta}}^2 \asymp |u^0|_{D(A_1^{(1+n\theta)/2})}^2 + |u^1|_{D(A_1^{n\theta/2})}^2 + |v^0|_{D(A_2^{(1 + n\theta)/2})}^2 + |v^1|_{D(A_2^{n\theta/2})}^2\, ,
$$
where $\asymp$ stands for the equivalence between norms.
\end{cor}

\section{Applications to PDEs}\label{se:applications}

In this section we describe some examples of systems of partial differential equations that can be studied by the results of this paper, but fail to satisfy the compatibility condition \eqref{theta2a}. We will hereafter denote by $\Omega$ a bounded domain in $\R^N$ with a sufficiently smooth boundary $\Gamma$. For $i=1,\dots,N$ we will denote by $\partial_i$ the partial derivative with respect to $x_i$ and by $\partial_t$ the derivative with respect to the time variable. We will also use the notation $H^k(\Omega),H^k_0(\Omega)$ for the usual Sobolev spaces with norm
$$
\|u\|_{k,\Omega}=\Big[\int_\Omega
\sum_{|p|\le k}
|D^p u|^2dx\Big]^{\frac 12}
\,,
$$
where we have set $D^p=\partial_1^{p_1}\cdots\partial_N^{p_N}$ for any multi-index
$p=(p_1,\dots,p_N)$.
Finally, we will denote by $C_\Omega>0$ the largest constant such that Poincar\'e's inequality 
\begin{equation}\label{poincare}
 C_\Omega\|u\|^2_{0,\Omega}\le \|\nabla u\|^2_{0,\Omega}
\end{equation}
holds true for any $u\in H^1_0(\Omega)$. To avoid too many notation, we denote in the same way the constant $C_\Omega$ such that

\begin{equation}\label{robinpoincare}
 C_\Omega\|u\|^2_{0,\Omega}\le \|\nabla u\|^2_{0,\Omega}+ \|u\|^2_{0, \Gamma}
\,,
\end{equation}
\noindent for all $u \in H^1{\Omega}$. In the following examples we take
\begin{equation*}\label{cond:set}
H = L^2(\Omega)\, ,\ B = \beta I\, .
\end{equation*}

\begin{example}\label{ex1}
Let $\beta,\,\lambda > 0,\,\a\in\mathbb{R}$, and consider the problem
\begin{equation}\label{sistex1}
\begin{array}{ll}
\begin{cases}
\partial^2_t u - \Delta u + \beta\partial_t u + \lambda u + \a v = 0 \\
\partial^2_t v -\Delta v + \a u = 0
\end{cases}
& \text{in}\ \Omega\times (0,+\infty)
\end{array}
\end{equation}
with boundary conditions
\begin{eqnarray}\label{bcex1}
\frac{\partial u}{\partial\nu}(\cdot,t) = 0\ \text{on}\ \Gamma\, ,\qquad v(\cdot,t) = 0\ \text{on}\ \Gamma\quad\forall t>0
\end{eqnarray}
and initial conditions
\begin{equation}\label{eq:initcond}
\begin{array}{ll}
\begin{cases}
u(x,0) = u^0(x) & u'(x,0) = u^1(x) \\
v(x,0) = v^0(x) & v'(x,0) = v^1(x)
\end{cases}
& x\in \Omega\, .
\end{array}
\end{equation}
The above system can be rewritten in abstract form taking
\begin{equation}
\begin{array}{c}
\displaystyle D(A_1) = \left\{u\in H^2(\Omega) : \frac{\partial u}{\partial\nu} = 0 \ \text{on}\ \Gamma\right\}\, ,\quad A_1 u = -\Delta u + \lambda u \, , \\
\displaystyle D(A_2) = H^2(\Omega)\cap H^1_0(\Omega)\, ,\quad A_2 v = -\Delta v\, .
\end{array}
\end{equation}
Notice that, in order to verify assumption $(H3)$, we shall choose $\a$ such that $0< |\a| < \left(C_\Omega(C_\Omega + \lambda)\right)^{1/2}$. Then,
\begin{multline*}
\left|\left< A_1 u,v\right>\right| = \displaystyle\left|\int_{\Omega}\nabla u \nabla v\, dx +  \lambda \int_{\Omega} u v\, dx\right| \\
 \le \displaystyle\left(\int_\Omega|\nabla u|^2\,dx\right)^{1/2} \left(\int_\Omega|\nabla v|^2\,dx\right)^{1/2} + \lambda\left(\int_\Omega u^2dx\right)^{1/2}\left(\int_\Omega v^2dx\right)^{1/2} \\
 \le \displaystyle c \left< A_1 u,u\right>^{1/2} |A_2 v|\, ,
\end{multline*}
where we have used the coercivity of $A_2$ and the well-known inequality
$$
\displaystyle\int_\Omega v^2+|\nabla v|^2\,dx \le c \int_\Omega|\Delta v|^2\,dx\quad \forall\, v\in H^2(\Omega)\cap H^1_0(\Omega)\, .
$$
Since condition \eqref{eqnequiv} is fulfilled, we get the following conclusions.

$i)$ If $(u^0,u^1,v^0,v^1) \in D(A_1)\times D(A_1^{1/2})\times D(A_2)\times D(A_2^{1/2})$, then the solution $U$ of problem \eqref{sistex1}-\eqref{bcex1}-\eqref{eq:initcond} satisfies 
\begin{equation}\label{eq:maind3}
E_1(u(t),u'(t)) + E_2(v(t),v'(t))
\le \frac{c}{t^{1/4}} \|U_0\|_{D({\mathcal A})}^2\qquad \forall t>0
\end{equation}
for some constant $c>0$. Moreover, there exists $c_1 > 0$ such that
$$
\|U_0\|_{D({\mathcal A})}^2 \le c_1\left(\|u^0\|_{2,\Omega}^2 + \|u^1\|^2_{1,\Omega} + \|v^0\|^2_{2,\Omega} + \|v^1\|^2_{1,\Omega}\right)\, .
$$

$ii)$ By point $ii)$ of Corollary \ref{th:main*}, if $U_0\in (\mathcal{H},D(\mathcal{A}^n))_{\theta,2}$ for some $0<\theta < 1$, $n\ge 1$, then the solution of \eqref{sistex1}-\eqref{bcex1}-\eqref{eq:initcond} satisfies 
\begin{equation}
E_1(u(t),u'(t)) + E_2(v(t),v'(t))
\le \frac{c_{n,\theta}}{t^{n\theta/4}}\|U_0\|_{(\mathcal{H},D(\mathcal{A}^n))_{\theta,2}}^2
\end{equation}
for every $t>0$ and some constant $c_{n,\theta} >0$. Moreover, point $iii)$ of Corollary~\ref{th:main*} ensures that, if $U_0\in D((-\mathcal{A})^\theta)$ for some $0<\theta < 1$, then
\begin{equation}
E_1(u(t),u'(t)) + E_2(v(t),v'(t))
\le \frac{c_{\theta}}{t^{\theta/4}}\|U_0\|_{D((-\mathcal{A})^\theta)}^2
\qquad \forall t>0
\end{equation}
for some constant $c_{\theta} >0$.

Of interest is the case when an operator fulfills different boundary conditions on proper subsets of $\Gamma$. For instance, let $\Gamma_0$ be an open subset of $\Gamma$ (with respect to the topology of $\Gamma$) and set $\Gamma_1=\Gamma \backslash \Gamma_0$. We assume that $\overline{\Gamma_0}\cap \overline{\Gamma_1} =\emptyset$. Consider the system \eqref{sistex1} with boundary conditions
\begin{equation}\label{bcex6}
\begin{array}{rl}
\displaystyle u(\cdot,t)  = 0\ \text{on}\ \Gamma_0\, ,&\ \displaystyle\frac{\partial u}{\partial\nu}(\cdot,t) = 0 \ \text{on } \Gamma\setminus \Gamma_0 \\
\displaystyle v(\cdot,t) & = 0\ \text{on}\ \Gamma
\end{array}
\qquad \forall t>0
\end{equation}
and initial conditions \eqref{eq:initcond}. Let us set
\begin{equation*}
\begin{array}{c}
\displaystyle 
D(A_1) = \left\{u\in H^2(\Omega): u  = 0\ \text{on}\ \Gamma_0\, ,\ \frac{\partial u}{\partial\nu} = 0 \ \text{on } \Gamma\setminus \Gamma_0\right\}\, , \\
A_1 u = -\Delta u\, .
\end{array}
\end{equation*}
Then, $\left|\left< A_1 u,v\right>\right| \le c \left< A_1 u,u\right>^{1/2} |A_2 v|$.
So, for $0< |\a| < \left(C_\Omega(C_\Omega + \lambda)\right)^{1/2}$, condition \eqref{eq:theta} is fulfilled, and the same conclusions $i)-ii)$ hold for problem \eqref{sistex1}-\eqref{bcex6}-\eqref{eq:initcond}.
\end{example}

\begin{example}\label{ex2}
Another interesting situation occurs while coupling two equa\-tions of different orders. Let $\beta, \lambda > 0,\,\a\in\mathbb{R}$, and consider the system
\begin{equation}\label{sistex2}
\begin{array}{ll}
\begin{cases}
\partial^2_t u + \Delta^2 u + \lambda u + \beta\partial_t u + \a v = 0 \\
\partial^2_t v -\Delta v + \a u = 0
\end{cases}
& \text{in}\ \Omega\times (0,+\infty)
\end{array}
\end{equation}
with boundary conditions
\begin{eqnarray}\label{bcex2}
\Delta u(\cdot,t) = 0 = \frac{\partial \Delta u}{\partial\nu}(\cdot,t)\ \text{on}\ \Gamma\, ,\qquad v(\cdot,t) = 0\ \text{on}\ \Gamma\quad \forall t>0
\end{eqnarray}
and initial conditions \eqref{eq:initcond}. Define
\begin{equation*}\label{operatorsex2}
\begin{array}{c}
\displaystyle D(A_1) = \left\{u\in H^4(\Omega) : \Delta u = 0 = \frac{\partial \Delta u}{\partial\nu}\ \text{on}\ \Gamma\right\}\, ,\quad A_1 u = \Delta^2 u + \lambda u\, , \\
\displaystyle D(A_2) = H^2(\Omega) \cap H^1_0(\Omega)\, ,\quad A_2 v = -\Delta v\, .
\end{array}
\end{equation*}
Suppose $0< |\a| < \lambda^{1/2} C_\Omega^{1/2}$, as required by $(H3)$. Observing that, for any $u\in D(A_1)$ and $v\in D(A_2)$,
\begin{equation*}
\left|\left< A_1 u,v\right>\right| = \displaystyle\left|\int_{\Omega}\Delta u \Delta v\, dx + 
\lambda\int_{\Omega} u v\, dx \right| 
 \le \displaystyle c \left< A_1 u,u\right>^{1/2} |A_2 v|\, ,
\end{equation*}
we conclude that condition \eqref{eqnequiv} is fulfilled. So, for every $U_0 \in D({\mathcal A})$, the solution $U$ of problem \eqref{sistex2}-\eqref{bcex2}-\eqref{eq:initcond} satisfies
\begin{equation}
E_1(u(t),u'(t)) + E_2(v(t),v'(t))
\le \frac{c}{t^{1/4}} \|U_0\|_{D({\mathcal A})}^2\qquad \forall t>0
\end{equation}
for some constant $c>0$. Moreover, there exists $c_1 > 0$ such that
$$
\|U_0\|_{D({\mathcal A})}^2 \le c_1\left(\|u^0\|_{4,\Omega}^2 + \|u^1\|^2_{2,\Omega} + \|v^0\|^2_{2,\Omega} + \|v^1\|^2_{1,\Omega}\right)\, .
$$
Note that we give in Example~\ref{xe1} another set of boundary conditions for the same symbols for the operators. It is interesting to see that both examples are treated for different classes of compatibility conditions, namely the present example satisfies the compatibility condition \eqref{eq:thetab}, whereas the example~\eqref{xe1} satisfies the compatibility condition \eqref{theta2a}. 
\end{example}

\begin{example}\label{ex3}
Let $\beta > 0$, $\a\in\R$, and consider the problem
\begin{equation}\label{eq:ex3}
\begin{array}{ll}
\begin{cases}
\partial^2_t u - \Delta u + \beta \partial_t u + \alpha v = 0 \\
\partial^2_t v -\Delta v + \alpha u = 0
\end{cases}
& \text{in}\ \Omega\times (0,+\infty)
\end{array}
\end{equation}
with boundary conditions
\begin{equation}\label{bcex3}
\begin{array}{rl}
\displaystyle\left(\frac{\partial u}{\partial\nu} + u\right)(\cdot,t)  &= 0\ \text{on}\ \Gamma \\
v(\cdot,t) & = 0\ \text{on}\ \Gamma
\end{array}
\qquad \forall t>0
\end{equation}
and initial conditions \eqref{eq:initcond}. Let us define
\begin{equation}\label{opex3}
\!\!\!\!\begin{array}{l}
\displaystyle D(A_1) = \left\{u\in H^2(\Omega): \frac{\partial u}{\partial\nu} + u = 0 \ \text{on}\ \Gamma\right\}\, ,\ A_1 u = -\Delta u\, , \\
\displaystyle D(A_2) = H^2(\Omega)\cap H^1_0(\Omega)\, ,\ A_2 v = -\Delta v\, ,
\end{array}
\end{equation}
and assume $0< |\a| < C_\Omega$. Observe that
\begin{multline*}
\left|\left< A_1 u,v\right>\right| = \displaystyle\left|\int_{\Omega}\nabla u \nabla v\, dx \right| \\
\le \displaystyle\left(\int_\Omega|\nabla u|^2\,dx\right)^{1/2} \left(\int_\Omega|\nabla v|^2\,dx\right)^{1/2} \le \displaystyle c \left< A_1 u,u\right>^{1/2} |A_2 v|\, ,
\end{multline*}
since
\begin{eqnarray*}
\left< A_1 u,u\right> = \int_\Omega|\nabla u|^2\,dx + \int_\Gamma |u|^2\,dS\, ,\qquad
\displaystyle \int_\Omega|\nabla v|^2\,dx \le c \int_\Omega|\Delta v|^2\,dx\, .
\end{eqnarray*}
Thus, condition \eqref{eq:theta} is fulfilled. So, the energy of the solution of problem \eqref{eq:ex3}-\eqref{bcex3}-\eqref{eq:initcond} satisfies
\begin{equation}
E_1(u(t),u'(t)) + E_2(v(t),v'(t))
\le \frac{c}{t^{1/4}} \|U_0\|_{D({\mathcal A})}^2\qquad \forall t>0
\end{equation}
for some constant $c>0$. Moreover, there exists $c_1 > 0$ such that
$$
\|U_0\|_{D({\mathcal A})}^2 \le c_1\left(|A_1 u^0|^2 + |A_1^{1/2}u^1|^2 + |A_2 v^0|^2 + |A_2^{1/2}v^1|^2\right)\, .
$$
\end{example}
Our next result show that the operators in Example \ref{ex3} do not fulfill the compatibility condition \eqref{theta2a}.

\begin{prop}\label{compcondns}
Let $A_1$, $A_2$ be defined as in \eqref{opex3}. Then for every $k\in \mathbb{N}$, $k\ge 2$, $D(A_2^{k/2})$ is not included in $D(A_1)$.
\end{prop}

\begin{proof}
Since $D(A_2^k)\subset D(A_2^{k/2})$ for every $k\in \mathbb{N}$, it is sufficient to prove that $D(A_2^k)$ is not included in $D(A_1)$ for every $k\in \mathbb{N}$, $k\ge 1$. For this purpose, let us fix $k\in \mathbb{N}$, $k\ge 1$, and consider the problem
\begin{equation}
\begin{cases}
(-\Delta)^k v_0 = 1 \\
v_0 = 0 = \Delta v_0 = \dots = \Delta^{k-1}v_0 \qquad \text{on } \Gamma\, .
\end{cases}
\end{equation}
Define the sequence $v_1,\, v_2,\,\dots, v_{k-1}$ by
\begin{equation}\label{eq:lapdom1}
\begin{array}{lll}
\begin{cases}
-\Delta v_0 = v_1 \\
v_{0_{\mid\Gamma}} = 0
\end{cases}
&
\dots\quad 
\begin{cases}
-\Delta v_{k-2} = v_{k-1} \\
v_{k-2_{\mid\Gamma}} = 0
\end{cases}
&
\begin{cases}
-\Delta v_{k-1} = 1 \\
v_{k-1_{\mid\Gamma}} = 0 \, .
\end{cases}
\end{array}
\end{equation}
We will argue by contradiction, assuming $D(A_2^{k})\subset D(A_1)$. Since $v_0$ belongs to $D(A_2)\cap D(A_1)$, we have $v_{0_{\mid\Gamma}} = 0 = \frac{\partial v_0}{\partial \nu}_{\mid\Gamma}$. Moreover, from the first system in \eqref{eq:lapdom1}, it follows that
\[
\int_\Omega v_1 dx = \int_\Omega (-\Delta v_0) dx = - \int_\Gamma \frac{\partial v_0}{\partial \nu} dS = 0\, .
\]
Hence, $\displaystyle\int_\Omega v_1 dx = 0$. Let us prove by induction that
\begin{equation}\label{eqcond1}
\int_\Omega \nabla v_{k-i}\nabla v_i dx = 0\, \quad\forall\, i = 1,\, 2,\dots , k-1\, .
\end{equation}
For $i = 1$ we have
$$
\int_\Omega \nabla v_{k-1}\nabla v_1 dx = \int_\Omega (-\Delta v_{k-1}) v_1 dx = \int_\Omega v_1 dx = 0\, ,
$$
since $v_{k-1_{\mid\Gamma}} = 0 = v_{1_{\mid\Gamma}}$. Now, let $i>1$ and suppose
$$
\int_\Omega \nabla v_{k-i}\nabla v_i dx = 0\, .
$$
Then,
$$
\begin{array}{rl}
0 = & \displaystyle\int_\Omega v_{k-i}(-\Delta v_i) dx = \int_\Omega v_{k-i}v_{i+1} dx \\
 = & \displaystyle\int_\Omega (-\Delta v_{k-i-1}) v_{i+1} dx = \int_\Omega \nabla v_{k-(i+1)}\nabla v_{i+1} dx \, .
\end{array}
$$
Thus, \eqref{eqcond1} holds for $i + 1$. Moreover, from \eqref{eqcond1} follows that
\begin{equation}\label{eqcond2}
\int_\Omega v_{k-i} v_{i+1} dx = 0\quad \forall\, i = 1,\, 2,\dots , k-1\, ,
\end{equation}
since
$$
\int_\Omega v_{k-i}v_{i+1} dx = \int_\Omega v_{k-i} (-\Delta v_{i}) dx = \int_\Omega \nabla v_{k-i}\nabla v_{i} dx = 0 \, .
$$
Now, let $k$ be even, say $k=2p$, $p\in\mathbb{N}^*$. Then, by \eqref{eqcond1} with $i=p$ we obtain
$$
\int_\Omega |\nabla v_p|^2 dx = 0\, ,\ \textrm{whence } v_p = 0\, .
$$
So, by a cascade effect,
$$
v_{p+1} = -\Delta v_p = 0\Rightarrow v_{p+2} = -\Delta v_{p+1} = 0 \Rightarrow \dots \Rightarrow v_{k-1} = -\Delta v_{k-2} = 0\, .
$$
Since $-\Delta v_{k-1} = 1$, we get a contradiction. If, on the contrary, $k$ is odd, i.e. $k = 2p+1$, then, applying \eqref{eqcond2} with $i=p$, we conclude that
$$
\int_\Omega |v_{p+1}|^2 dx = 0\, ,\ \textrm{whence } v_{p+1} = 0\, .
$$
Finally, we have that $v_{p+1} = v_{p+2} = \dots = v_{k-1} = 0$. Since $-\Delta v_{k-1} = 1$, we get a contradiction again. Therefore, $D(A_2^{k})$ is not included in $D(A_1)$.
\end{proof}

\begin{example}\label{ex4}
Given $\beta > 0$, $\a\in\R$, let us now consider the undamped Petrowsky equation coupled with the damped wave equation,
\begin{equation}\label{sistex4}
\begin{array}{ll}
\begin{cases}
\partial^2_t u - \Delta u + \beta\partial_t u + \alpha v = 0 \\
\partial^2_t v + \Delta^2 v + \alpha u = 0
\end{cases}
& \text{in}\ \Omega\times (0,+\infty)
\end{array}
\end{equation}
with Robin boundary conditions
\begin{equation}\label{bcex4u}
\displaystyle\left(\frac{\partial u}{\partial\nu} + u\right)(\cdot,t) = 0\ \text{on}\ \Gamma \qquad \forall t>0
\end{equation}
on $u$ and either
\begin{equation}\label{bcex4v1}
\displaystyle v(\cdot,t) = \Delta v(\cdot,t) = 0\ \text{on}\ \Gamma \qquad \forall t>0
\end{equation}
or
\begin{equation}\label{bcex4v2}
\displaystyle v(\cdot,t) = \frac{\partial v}{\partial\nu}(\cdot,t) = 0\ \text{on}\ \Gamma \qquad \forall t>0
\end{equation}
on $v$, with initial conditions \eqref{eq:initcond}. Define
\begin{equation*}\label{opex4}
\begin{array}{c}
\displaystyle D(A_1) = \left\{u\in H^2(\Omega): \frac{\partial u}{\partial\nu} + u = 0 \ \text{on}\ \Gamma\right\}\, ,\quad A_1 u = -\Delta u\, , \\
\displaystyle D(A_2) = \left\{v\in H^4(\Omega): v = \Delta v = 0 \ \text{on}\ \Gamma\right\}\, ,\quad A_2 v = \Delta^2 v
\end{array}
\end{equation*}
(with boundary conditions \eqref{bcex4v1} on $v$), or
\begin{equation*}
\displaystyle \tilde{D}(A_2) = \left\{v\in H^4(\Omega): v = \frac{\partial v}{\partial\nu} = 0 \ \text{on}\ \Gamma\right\}\, ,\quad A_2 v = \Delta^2 v
\end{equation*}
(with boundary conditions \eqref{bcex4v2} on $v$). Once again, we have
\begin{multline*}
\left|\left< A_1 u,v\right>\right| = \displaystyle\left|\int_{\Omega}\nabla u \nabla v\, dx \right| \\
\le \displaystyle\left(\int_\Omega|\nabla u|^2\,dx\right)^{1/2} \left(\int_\Omega|\nabla v|^2\,dx\right)^{1/2} \le \displaystyle c \left< A_1 u,u\right>^{1/2} |A_2 v|\, .
\end{multline*}
Thus, condition \eqref{eq:theta} is fulfilled and, for $0 < |\a| < C_\Omega^{3/2}$, the polynomial decay of the energy of solution to \eqref{sistex4}-\eqref{bcex4u}-\eqref{bcex4v1}-\eqref{eq:initcond} and \eqref{sistex4}-\eqref{bcex4u}-\eqref{bcex4v2}-\eqref{eq:initcond} follows as in Example \ref{ex1}.
\end{example}

\section{Improvement of previous results}\label{impr}

In this section we apply interpolation theory to extend the polynomial stability result of \cite{alcako02} to a larger class of initial data. We will denote by $j\ge 2$ the integer for which \eqref{theta2a} is satisfied.
As is shown in \cite[Theorem 4.2]{alcako02}, under assumptions $(H1),(H2),(H3)$ and \eqref{theta2a}, if $U_0\in D({\mathcal A}^{nj})$ for some integer $n\ge 1$, the solution $U$ of problem \eqref{eq:2z} satisfies 
\begin{equation}\label{3.2ter}
{\mathcal E}(U(t))
\le \frac{c_n}{t^n}\sum_{k=0}^{nj}{\mathcal E}(U^{(k)}(0))
\qquad \forall t>0
\end{equation}
for some constant $c_n>0$. We recall that assumption \eqref{theta2a} covers many si\-tua\-tions of interest for applications to systems of evolution equations. Indeed (see \cite[Section 5]{alcako02} for further details), this is the case for
\begin{itemize}
\item[i)] $(A_1,D(A_1)) =(A_2,D(A_2))$, where \eqref{theta2a} is fulfilled with $j=2$;
\item[ii)] $D(A_1) = D(A_2)$, with $j = 2$;
\item[iii)] $(A_2,D(A_2)) = (A_1^2,D(A_1^2))$, again with $j=2$;
\item[iv)] $(A_1,D(A_1)) = (A_2^2,D(A_2^2))$, with $j=4$.
\end{itemize}
The following result completes the analysis of \cite{alcako02}, taking the initial data in suitable interpolation spaces.

\begin{thm}\label{intresack}
Assume $(H1),(H2),(H3)$ and \eqref{theta2a}, and let $0<\theta < 1$, $n\ge 1$.
Then for every $U_0$ in $(\mathcal{H},D(\mathcal{A}^{nj}))_{\theta,2}$, the solution $U$ of \eqref{eq:2z}
satisfies 
\begin{equation}\label{eq:mainterdack}
\|U(t)\|_{\mathcal{H}}^2
\le \frac{c_{n,\theta}}{t^{n\theta}}\|U_0\|_{(\mathcal{H},D(\mathcal{A}^{nj}))_{\theta,2}}^2
\qquad \forall t>0
\end{equation}
for some constant $c_{n,\theta}>0$.
\end{thm}
Reasoning as in Remark \ref{rmkcnst}, one can derive estimate \eqref{3.2ter} also for $U_0\in D(\mathcal{A}^k)$, for every $k=1, \dots, nj - 1$, with decay rate $k/j$.

\begin{cor}\label{th:mainack*}
Assume $(H1),(H2),(H3)$ and $(\ref{theta2a})$.
\begin{itemize}
\item[i)] If $U_0\in D({\mathcal A}^{n})$ for some $n\ge 1$, then the solution of $(\ref{eq:2z})$ satisfies
\begin{equation}\label{eq:maindack*}
\|U(t)\|_\mathcal{H}^2
\le \frac{c_n}{t^{n/j}}\|U_0\|_{D(\mathcal{A}^n)}^2
\qquad \forall t>0
\end{equation}
for some constant $c_n>0$.
\item[ii)] If $U_0\in (\mathcal{H},D(\mathcal{A}^n))_{\theta,2}$ for some $n\ge 1$ and $0<\theta < 1$, then the solution of $(\ref{eq:2z})$ satisfies 
\begin{equation}\label{eq:maind2ack*}
\|U(t)\|_{\mathcal{H}}^2
\le \frac{c_{n,\theta}}{t^{n\theta/j}}\|U_0\|_{(\mathcal{H},D(\mathcal{A}^n))_{\theta,2}}^2
\qquad \forall t>0
\end{equation}
for some constant $c_{n,\theta}>0$.
\item[iii)] If $U_0\in D((-\mathcal{A})^\theta)$ for some $0<\theta < 1$, then the solution of problem \eqref{eq:2z}
satisfies
\begin{equation}\label{eq:intermimp}
\|U(t)\|_{\mathcal{H}}^2
\le \frac{c_{\theta}}{t^{\theta/j}}\|U_0\|_{D((-\mathcal{A})^\theta)}^2
\qquad \forall t>0
\end{equation}
for some constant $c_{\theta}>0$.
\end{itemize}
\end{cor}

In particular, the previous fractional decay rates can be achieved for initial data in $\mathcal{H}_{n}$ or in $\mathcal{H}_{n,\theta}$, whenever $\mathcal{H}_{n}\subset D({\mathcal A}^n)$, as in Corollary \ref{th:resintspbetbound}. This happens, for instance, if any of the following conditions is satisfied:
\begin{itemize}
\item[i)] $(A_1,D(A_1)) =(A_2,D(A_2))$;
\item[ii)] $D(A_1) = D(A_2)$;
\item[iii)] $(A_2,D(A_2)) = (A_1^2,D(A_1^2))$.
\end{itemize}

Let us apply Corollary \ref{th:mainack*} to two examples from \cite{alcako02}.

\begin{example}
Given $\beta>0,\,\kappa > 0,\,\a\in\mathbb{R}$, let us study the problem
\begin{equation}\label{WW}
\left\{\begin{array}{l}
\partial_t ^2u-\Delta u+\beta\partial_t u+\kappa u + \a v = 0 
\\
\partial_t ^2v -\Delta  v+ \kappa v  + \a  u =0 
\end{array}\right.
\qquad\mbox{in}\qquad\Omega\times (0,+\infty)
\end{equation}
with boundary conditions
\begin{equation}\label{DBC}
u(\cdot,t)=0=v(\cdot,t)\qquad \mbox{on}\quad\Gamma\quad\forall t>0
\end{equation}
and initial conditions
\begin{equation}\label{ic2}
\left\{\begin{array}{ll}
\quad u(x,0)=u^0(x)\,,&\quad u'(x,0)=u^1(x)
\\
\quad v(x,0)=v^0(x)\,,&\quad v'(x,0)=v^1(x)
\end{array}\right.
\qquad x\in\Omega\,.
\end{equation}
Let $H=L^2(\Omega)$, $B=\beta I$, and $A_1 = A_2 = A$ be defined by 
$$
D(A)=H^2(\Omega)\cap H^1_0(\Omega)\,,\qquad
Au=-\Delta u+\kappa u\quad\forall u\in D(A)\,.
$$
Notice that \eqref{theta2a} is fulfilled with $j=2$, and condition $0<|\a|<C_\Omega+\kappa=:\omega$ is required in order to fulfill $(H3)$.

As showed in \cite[Example 6.1]{alcako02}, if $u^0,v^0\in H^2(\Omega)\cap H^1_0(\Omega)$ and $u^1,v^1\in H^1_0(\Omega)$, then
\begin{displaymath}
\begin{array}{l}
\displaystyle \int_\Omega
\Big(|\partial_t u|^2+|\nabla u|^2+|\partial_t v|^2+|\nabla v|^2\Big)dx \\
\qquad\qquad \displaystyle \le \frac{c}{t}
\Big(\|u^0\|^2_{2,\Omega}+ \|u^1\|^2_{1,\Omega}+\|v^0\|^2_{2,\Omega}+ \|v^1\|^2_{1,\Omega} \Big)
\qquad\forall t>0\,.
\end{array}
\end{displaymath}
Moreover, if $u^0,v^0\in H^{n+1}(\Omega)$ and $u^1,v^1\in H^n(\Omega)$ are such that
$$
u^0=\dots =\Delta^{[\frac{n}{2}]}u^0=0=v^0=\dots =\Delta^{[\frac{n}{2}]}v^0\quad \mbox{on}\quad\Gamma,
$$
$$
u^1=\dots =\Delta^{[\frac{n-1}{2}]}u^1=v^1=\dots =\Delta^{[\frac{n-1}{2}]}v^1=0\quad\mbox{on}\quad\Gamma,
$$
then 
\begin{displaymath}
\begin{array}{l}
\displaystyle\int_\Omega\Big(|\partial_t u|^2+|\nabla u|^2+|\partial_t v|^2+|\nabla v|^2\Big)dx \\
\qquad \qquad \le \displaystyle
\frac{c_n }{t^n}
\Big(\|u^0\|^2_{n+1,\Omega}+ \|u^1\|^2_{n,\Omega}+\|v^0\|^2_{n+1,\Omega}+ \|v^1\|^2_{n,\Omega} \Big)
\qquad\forall t>0\,.
\end{array}
\end{displaymath}

Furthermore, applying Corollary \ref{th:mainack*}, we conclude that if $U_0$ belongs to $\mathcal{H}_{n,\theta} = (\mathcal{H},D(\mathcal{A}^n))_{\theta,2}$ for some $0 < \theta < 1$, $n\ge 1$, then the solution to \eqref{WW}-\eqref{DBC}-\eqref{ic2} satisfies
\begin{equation}
\displaystyle\int_\Omega\Big(|\partial_t u|^2+|\nabla u|^2+|\partial_t v|^2+|\nabla v|^2\Big)dx \le \displaystyle
\frac{c_{n,\theta}}{t^{n\theta/2}}\|U_0\|_{\mathcal{H}_{n,\theta}}^2
\qquad \forall t>0
\end{equation}
for some constant $c_{n,\theta}>0$, with
$$
\|U_0\|_{\mathcal{H}_{n,\theta}}^2 \asymp |u^0|_{D(A_1^{\frac{1}{2} + \frac{n}{2}\theta})}^2 + |u^1|_{D(A_1^{\frac{n}{2}\theta})}^2 + |v^0|_{D(A_2^{\frac{1}{2} + \frac{n}{2}\theta})}^2 + |v^1|_{D(A_2^{\frac{n}{2}\theta})}^2\, .
$$
\end{example}

\begin{example}\label{xe1}
Taking $\beta > 0,\, 0 < |\a| < C_\Omega^{3/2}$, and the same operators $A_1$ and $A_2$ as in Example~\ref{ex2}, but with different boundary conditions, we can consider the system 
\begin{equation}\label{WW2}
\left\{\begin{array}{l}
\partial_t ^2u+\Delta^2 u+\beta\partial_t u+\alpha v=0 
\\
\partial_t ^2v -\Delta  v+ \alpha u =0 
\end{array}\right.
\qquad\mbox{in}\qquad\Omega\times (0,+\infty)
\end{equation}
with boundary conditions
\begin{equation}\label{DBC2}
v(\cdot,t)=u(\cdot,t)=\Delta u(\cdot,t)=0\qquad \mbox{on}\quad\Gamma\quad\forall t>0
\end{equation}
and initial conditions as in \eqref{ic2}. 
Let us set $H=L^2(\Omega)$, $B=\beta I$, and 
\begin{equation*}
\begin{array}{c}
\displaystyle D(A_1) = \left\{u\in H^4(\Omega) : \Delta u = 0 = u\ \text{on}\ \Gamma\right\}\, ,\quad A_1 u = \Delta^2 u\, , \\
\displaystyle D(A_2) = H^2(\Omega)\cap H^1_0(\Omega)\, ,\quad A_2 v = -\Delta v\, .
\end{array}
\end{equation*}
In this case, since $A_1=A_2^2$, condition (\ref{theta2a}) holds
with
$j=4$. Consequently, as is shown in \cite[Example 6.4]{alcako02}, for initial condition $U_0\in D({\mathcal A}^4)$
$$
\int_\Omega
\Big(|\partial_t u|^2+|\Delta u|^2+|\partial_t v|^2+|\nabla v|^2\Big)dx
\le\frac{C}{t} \|U_0\|_{D({\mathcal A}^4)}^2
\qquad \forall t>0\, ,
$$
for some constant $C>0$. By point $i)$ of Corollary \ref{th:mainack*}, we can generalize this result to initial data $U_0\in D({\mathcal A}^{n})$ for some $n\ge 1$. Indeed, in this case the solution to \eqref{WW2}-\eqref{DBC2}-\eqref{ic2} satisfies
$$
\int_\Omega
\Big(|\partial_t u|^2+|\Delta u|^2+|\partial_t v|^2+|\nabla v|^2\Big)dx
\le \frac{c_n}{t^{n/4}} \|U_0\|_{D({\mathcal A}^n)}^2
\qquad \forall t>0\, ,
$$
for some constant $c_n>0$. Moreover, thanks to point $ii)$ of Corollary~\ref{th:mainack*}, if $U_0\in (\mathcal{H},D(\mathcal{A}^n))_{\theta,2}$ for some $n\ge 1$ and $0<\theta < 1$, then
$$
\int_\Omega
\Big(|\partial_t u|^2+|\Delta u|^2+|\partial_t v|^2+|\nabla v|^2\Big)dx
\le \frac{c_{n,\theta}}{t^{n\theta/4}}\|U_0\|_{(\mathcal{H},D(\mathcal{A}^n))_{\theta,2}}^2
\qquad \forall t>0
$$
for some constant $c_{n,\theta}>0$. Furthermore, thanks to point $iii)$ of Corollary~\ref{th:mainack*}, if $U_0$ belongs to $\mathcal{H}_{1,\theta} = D((-\mathcal{A})^\theta)$ for some $0 < \theta < 1$, then the solution to \eqref{WW2}-\eqref{DBC2}-\eqref{ic2}
satisfies
\begin{equation}
\int_\Omega
\Big(|\partial_t u|^2+|\Delta u|^2+|\partial_t v|^2+|\nabla v|^2\Big)dx
\le \frac{c_{\theta}}{t^{\theta/4}}\|U_0\|_{D((-\mathcal{A})^\theta)}^2
\qquad \forall t>0
\end{equation}
for some constant $c_{\theta}>0$, with
$$
\|U_0\|_{D((-\mathcal{A})^\theta)}^2 \asymp |u^0|_{D(A_1^{\frac{1}{2} + \frac{1}{2}\theta})}^2 + |u^1|_{D(A_1^{\frac{1}{2}\theta})}^2 + |v^0|_{D(A_2^{\frac{1}{2} + \frac{1}{2}\theta})}^2 + |v^1|_{D(A_2^{\frac{1}{2}\theta})}^2\, .
$$
\end{example}

\section*{Acknowledgments} We are grateful to the referees for their valuable comments and suggestions.


\medskip
Received xxxx 20xx; revised xxxx 20xx.
\medskip

\end{document}